\documentclass{scrartcl}
\usepackage[utf8]{inputenc}
\usepackage[T1]{fontenc}
\usepackage{amsmath,amsthm,amssymb}
\usepackage{newunicodechar}
\usepackage{nicefrac}
\newunicodechar{ε}{\varepsilon}
\newunicodechar{Δ}{\Delta}
\newunicodechar{ℝ}{\mathbb{R}}
\newunicodechar{ℕ}{\mathbb{N}}
\newunicodechar{∂}{\partial}
\newunicodechar{→}{\to}
\newunicodechar{ν}{\nu}
\newunicodechar{∞}{\infty}
\newunicodechar{φ}{\varphi}
\newunicodechar{ψ}{\psi}
\newunicodechar{θ}{\theta}
\newunicodechar{Φ}{\Phi}
\newunicodechar{ϕ}{\phi}
\newunicodechar{ρ}{\rho}
\newunicodechar{σ}{\sigma}
\newunicodechar{π}{\pi}
\newunicodechar{τ}{\tau}
\newunicodechar{χ}{\chi}
\newunicodechar{α}{\alpha}
\newunicodechar{β}{\beta}
\newunicodechar{ξ}{\xi}
\newunicodechar{η}{\eta}
\newunicodechar{δ}{\delta}
\newcommand{\norm}[2][]{\left\|#2\right\|_{#1}}
\newcommand{\normm}[2]{\left\|#2\right\|_{#1}}
\newcommand{\betrag}[1]{\left\lvert#1\right\rvert}

\newcommand{\xe}{{x_{ε}}}
\newcommand{\xet}{\xe_t}
\newcommand{\ye}{{y_{ε}}}
\newcommand{\yet}{\ye_t}

\newcommand{\f}[2]{\frac{#1}{#2}}

\newcommand{\Labi}{L^{∞}((a,b))}
\newcommand{\Lab}[1]{L^{#1}((a,b))}
\newcommand{\kl}[1]{\left(#1\right)}
\newcommand{\kkl}[1]{\left[#1\right]}
\newcommand{\set}[1]{\left\{#1\right\}}

\newcommand{\dt}{\,\mathrm{d} t}

\newcommand{\du}{\,\mathrm{d}u}
\newcommand{\dv}{\,\mathrm{d}v}
\newcommand{\ddt}{\f{d}{dt}}
\newcommand{\iab}{\int_a^b}
\newcommand{\calF}{\mathcal{F}}
\newcommand{\xhat}{\hat{x}}

\newcommand{\weakto}{\rightharpoonup}
\newcommand{\weakstarto}{\rightharpoonup_{\!\!\!\!\!\!*}\;}
\usepackage{xcolor}

\newtheorem{thm}{Theorem}[section]
\newtheorem{lemma}[thm]{Lemma}
\newtheorem{remark}[thm]{Remark}
\theoremstyle{definition}
\newtheorem{example}[thm]{Example}

\usepackage{authblk}

\title{Continuation beyond interior gradient blow-up in a semilinear parabolic equation}
\author{Marek Fila\footnote{e-mail: fila@fmph.uniba.sk} }
\author{Johannes Lankeit\footnote{e-mail: jlankeit@math.upb.de}}

\affil{\footnotesize Department of Applied Mathematics and Statistics, Comenius University,\\ Mlynsk\'a dolina, 84248 Bratislava, Slovakia}

\begin{document}

\maketitle 

\begin{abstract}
\noindent
\textbf{Abstract.} It is known that there is a class of semilinear parabolic equations for which interior gradient blow-up (in finite time) occurs for some solutions. We construct a continuation of such solutions after gradient blow-up. This continuation is global in time and we give an example when it never becomes a classical solution again.
\end{abstract}

\section{Introduction}
Many parabolic PDEs allow for solutions that ``blow up'', i.e. cease to exist at some time, at which some of their norms becomes unbounded. Here it is possible that the function itself remains bounded, while its spatial gradient is the quantity to become unbounded. In most of the situations where such gradient blow-up is known to occur, it takes place on the boundary of the domain (cf. \cite{BD,CG,D,FL,FTW,F,GH,K1,K2,LS,PS3,PS2,PS1,PZ,QR,S1,ZB,ZL}; see below for more details). Much less frequently, gradient blow-up in the interior of the domain has been observed. It has been shown to arise in some quasilinear parabolic equations \cite{G,AI}, but for even simpler, semilinear equations we are only aware of the examples in \cite{AF}. The class of problems considered there is a generalization of the prototypical example
\begin{align}\label{ueq}
\begin{cases}
 u_t=u_{xx} + f(u)u_x^3 \qquad &\text{in } (-1,1)\times(0,T),\\
 u(-1,t)=a, \qquad u(1,t)=b\qquad &\text{for all } t\in(0,T),\\
 u(\cdot,0)=u_0\qquad &\text{in } (-1,1)
 \end{cases}
\end{align}
for $a,b\in ℝ$, where $f$ can be chosen as $f(u)= u$, $u\in ℝ$, for example. 

While all of its solutions remain bounded, comparison arguments combined with travelling wave solutions show 
(see \cite[Sec. IV.41]{QS} for a different argument) that for suitable $a,b\in ℝ$ interior gradient blow-up can happen.
 By this we mean that there exist $x_0\in (-1,1)$, $T>0$ and sequences $(x_k)_{k\in ℕ}\to x_0$,  $(t_k)_{k\in ℕ}\nearrow T$
such that
\[
|u_x(x_k,t_k)|\to \infty \qquad \text{as}\quad k\to\infty.
\]

A natural question that arises is whether the corresponding solutions can be extended in some sense for $t>T$ and, perhaps, even may become classical solutions again. 

In order to shed light onto this issue, it turns out to be helpful to consider the problem in different coordinates and to investigate the 
function $x$ with $x(\cdot,t)$ being the inverse of $u(\cdot,t)$. (We will go into some detail deriving the equivalent initial boundary value  problem in Appendix \ref{sec:equivalence}.)

This leads us to be interested in the possibly degenerate problem
\begin{equation}\label{xeq}
 \begin{cases} 
    x_t=\displaystyle\f{x_{uu}-f(u)}{x_u^2}& \text{in } (a,b)\times(0,∞), \\
    x(a,t)=-1,\quad x(b,t)=1,& t\in (0,\infty),\\
    x(\cdot,0)=x_0& \text{in } [a,b],
 \end{cases}
\end{equation}
and our first result is concerned with its global solvability for $f\in C^2(ℝ)$.

Under the following conditions on the initial data
\begin{equation}\label{cond:x0}
 \begin{cases} x_0\in C^{2+β}([a,b]) \qquad \text{for some } β\in(0,1),\\
  x_0(a)=-1,\qquad x_0(b)=1,\\
  {x_0}_{uu}(a)=a, \qquad {x_0}_{uu}(b)=b,\\
  {x_0}_u>0 \qquad \text{in } [a,b],
 \end{cases}
\end{equation}
we will obtain:  
 
\begin{thm}\label{thm:existence}
 Let $a,b\in ℝ$ with $a<b$, $f\in C^2(ℝ)$ and let $x_0$ be as in \eqref{cond:x0}. Then there is a function $x\in C^0([a,b]\times[0,∞))$ which satisfies 
 \begin{align*}
  x\in C^0([0,∞);C^{1+α}([a,b])), \quad x_{uu}\in L^p_{loc}([a,b]\times[0,∞)), \quad 
  x_t\in L^{∞}((a,b)\times(0,∞))
 \end{align*}
  for any $p\in(1,∞)$ and $α\in(0,1)$ and satisfies \eqref{xeq} in a weak sense (see \eqref{equation:weakly} for a precise formulation).
\end{thm}

We can already note that -- in contrast to the solutions of the original problem \eqref{ueq} -- these solutions are global. This already answers the question about extensibility of solutions to \eqref{ueq}, with ``in a suitable sense'' being the change to \eqref{xeq}. 
We can even reveal the long time behaviour of these functions: 

\begin{thm}\label{thm:longterm}
 Let $a,b\in ℝ$ with $a<b$, $f\in C^2(ℝ)$ and let $x_0$ be as in \eqref{cond:x0}. 
 Then for the solution $x$ to \eqref{xeq} from Theorem~\ref{thm:existence} it holds that 
 \begin{equation}\label{eq:thm:conv}
  \lim_{t\to \infty} x(\cdot,t) = x^{∞}\qquad \text{ in } C^1([a,b]), 
 \end{equation}
 where $x^{∞}\in C^2([a,b])$ is defined by 
 \begin{equation}\label{def:xinfty-general}
  x^{∞}_{uu}(u):= f(u)\;\; \text{ for } u\in (a,b), \qquad x^{∞}(a)=-1, \qquad x^{∞}(b)=1.
 \end{equation}
\end{thm}

\begin{example}
 If $f(u)=u$, then 
 \begin{equation}\label{def:xinfty-example}
  x^{∞}(u):= \f{u^3}6+\f2{b-a} u - \f{b^2+ab+a^2}6 u -1-\f{2a}{b-a}+\f{ab^2+a^2b}6, \qquad u\in[a,b]. 
 \end{equation} 
\end{example}

Let us note that for $f\equiv 0$ the homogeneous Dirichlet and Neumann problems corresponding to 
the PDE from \eqref{xeq} were studied in \cite{CGK}. The solutions from \cite{CGK}  instantly become constant, no matter which initial data we impose. This is in contrast with Theorem~\ref{thm:existence} where nonhomogeneous Dirichlet boundary conditions are considered.\\ 

In order to address the second part of the question above -- whether these, now extended, solutions become solutions to the original problem again -- and to answer it by ``not necessarily'', let us turn our attention to a more specific, symmetric setting. 

We will assume that now  $-a=b\in(0,∞)$, $f(u)=u$, and that the initial data $x_0$, in addition to \eqref{cond:x0}, also satisfy  
\begin{equation}\label{cond:x0:special}
 \begin{cases}
  x_0(u)=-x_0(-u)\qquad &\text{ for all } u\in (-b,b),\\
  x_{0uu}(u)\le u\qquad &\text{ for all } u\in(0,b).
 \end{cases}
\end{equation}

In this situation we show that if the construction from Appendix \ref{sec:x-to-u} does not work at some time $t_0$, it will never again be applicable; and that furthermore such time does exist if $b$ is large enough: 

\begin{thm}\label{thm:special}
 Let $-a=b\in(\sqrt[3]{6},∞)$, $f(u)=u$ for every $u\in ℝ$, and assume that $x_0$ satisfies \eqref{cond:x0} and \eqref{cond:x0:special}. Let $x$ be the solution to \eqref{xeq} from Theorem~\ref{thm:existence}. Then there is $t_0>0$ such that $x_u>0$ in $[-b,b]\times(0,t_0)$, 
 but $x(\cdot,t)$ has no differentiable inverse for any $t\ge t_0$. \\
 Moreover, for large $t$, the image of $[0,b]$ under $x(\cdot,t)$ contains negative elements. 
\end{thm}

In terms of $u$, the scenario of Theorem~\ref{thm:special} corresponds to blow-up of the gradient $u_x$ in the interior of the domain (more precisely: at $x=0$).\\

%
%
%
%

\textbf{Boundary gradient blow-up.} 
As far as we know, the first example of gradient blow-up can be found in \cite{F}. For other early examples, which appeared around three decades later, we refer to \cite{CG,D,FL,K1,K2}. The spatial profile near a blow-up point on the boundary was studied in \cite{CG,FL,PS3} and the blow-up set in \cite{LS,PS3}. Results on the blow-up rate (in time) were established in \cite{CG,GH,PS1,ZL} and on continuation after blow-up in \cite{BD,FL,FTW,K2,PS3,PS2,PZ,QR}.

Our approach yields a continuation beyond gradient blow-up on the boundary which is completely different.

\begin{thm}\label{thm:bdry}
 Let $a=0$, $b\in (2,∞)$, $f\equiv 1$ and let $x_0$ be as in \eqref{cond:x0}. Then the global solution to \eqref{xeq} from Theorem~\ref{thm:existence} is such that 
 \[
  t_0:=\inf\set{t>0 \mid \exists u\in[0,b]: x_u(u,t)=0} \in (0,∞)
 \]
 and
 \begin{equation}\label{eq:nullstellevonxu}
  x_u(0,t_0)=0.
 \end{equation}
 Moreover, for large $t$, the image of $[0,b]$ under $x(\cdot,t)$ contains negative elements. 
\end{thm}

\begin{remark}
 As $x(0,t_0)=-1$ in the situation of Theorem~\ref{thm:bdry}, this corresponds to gradient blow-up at $x=-1$ (i.e. gradient blow-up on the boundary) in the problem 
 \[
  \begin{cases}
   u_t(x,t)=u_{xx}(x,t) + u_x^3(x,t), &\qquad x\in(-1,1),\; t>0,\\
   u(-1,t)=0, \quad u(1,t)=b,&\qquad t>0,\\
   u(\cdot,0)= x_0^{-1}&\qquad \text{in } (-1,1).
  \end{cases}
 \]
 Since $x$ is globally defined, letting $u(\cdot,t)$ be the (multi-valued) inverse of $x(\cdot,t)$ also for $t\ge t_0$, one obtains a continuation of the solution after gradient blow-up on the boundary. Note that the domain of $u(\cdot,t)$ does not remain restricted to $[-1,1]$; thus this continuation certainly differs from those from \textrm{\cite{BD,FL,FTW,K2,PS3,PS2,PZ,QR}} that do not have this property.
\end{remark}

\textbf{Plan of the paper. } In Section \ref{sec:existence} we will prove Theorem~\ref{thm:existence}. We will base our reasoning on the regularized problem 
\begin{align}\label{xepseq}
\begin{cases}
 \xe_t=\displaystyle\f{\xe_{uu}-f(u)}{\xe_u^2+ε}& \text{in }(a,b)\times(0,∞),\\
 \xe(a,t)=-1,\qquad \xe(b,t)=1&\text{for all } t\in(0,∞),\\
 \xe(\cdot,0)=x_0 & \text{in }(a,b),
\end{cases}
\end{align}
and rely on comparison arguments (see Lemma~\ref{lem:bound:xet}) as well as the energy functional (cf. Lemma~\ref{lem:energy}) 
\begin{equation}\label{def:F}
   \calF(x):= \f12 \iab x_u^2(u)\du + \iab f(u)x(u) \du,\qquad x\in X,  
\end{equation}
where 
\begin{equation}\label{def:X}
 X:=\set{ξ\in W^{1,2}((a,b))\mid ξ(a)=-1, ξ(b)=1}, 
\end{equation}
to derive suitable bounds (stated in Lemmata \ref{lem:bound:xet}, \ref{lem:boundsxe:h1andli} and \ref{lem:boundsxe:h2andcalpha}) for a passage to the limit $ε\searrow 0$ along a convenient sequence in Lemma~\ref{lem:conv}. 

Section \ref{sec:longterm} will then be devoted to the proof of Theorem~\ref{thm:longterm} and mainly rely on properties of $\calF$ and consequences of it being an energy functional. 

Section \ref{sec:where} will deal with the question where the derivative of a solution $x$ can vanish and where, hence, the equation may become singular. This information is directly applicable in the proof of Theorem~\ref{thm:bdry}. 

In Section \ref{sec:special}, finally, we will prove Theorem~\ref{thm:special}. Symmetry of the solution and Lemma~\ref{lem:firstatzero}, but also Theorem~\ref{thm:longterm} will be instrumental.

Before we begin, let us add a remark concerning the requirements on the initial data.
\begin{remark}
 At least for Sections \ref{sec:existence}, \ref{sec:longterm} (i.e. for Theorems \ref{thm:existence} and \ref{thm:longterm}), the condition ${x_0}_u>0$ in \eqref{cond:x0} could be replaced by the less natural but also less restrictive condition 
 \[
  \sup_{u\in[a,b]} \betrag{\f{{x_0}_{uu}(u)-f(u)}{{x_0}_u^2(u)}}<∞. 
 \]
\end{remark}

\section{Existence. Proof of Theorem~\ref{thm:existence}}\label{sec:existence}

The starting point of our analysis is an existence result for the regularized problem \eqref{xepseq}.
\begin{lemma}\label{lem:exeps}
 Let $a,b \in ℝ$ with $a<b$, $f\in C^2(ℝ)$ and let $x_0$ be as in \eqref{cond:x0}. Then for every $ε>0$, \eqref{xepseq} has a unique classical solution 
 \begin{equation}\label{regularity:xe}
  \xe \in C^{2,1}([a,b]\times[0,∞))\cap C^{4,2}([a,b]\times(0,∞)).
 \end{equation}
\end{lemma}
\begin{proof}
The classical parabolic theory in the form of \cite[Thm. VI.4.1]{LSU} provides $α>0$ and a unique solution $\xe\in C^{2+α,1+\f{α}2}([a,b]\times[0,∞))$ and successive applications of \cite[Thm. IV.5.2]{LSU} additionally ensure the higher regularity claimed in \eqref{regularity:xe}. 
\end{proof}

The following first boundedness information goes back to an observation already made in \cite{AF}:
\begin{lemma}\label{lem:bound:xet}
 Let $a,b\in ℝ$ with $a<b$, $f\in C^2(ℝ)$ and let $x_0$ be as in \eqref{cond:x0} and $ε>0$. Then, with $\xe$ from Lemma~\ref{lem:exeps},  
it holds that 
\[
  [0,∞)\ni t \mapsto \norm[\Labi]{\xet(\cdot,t)} 
 \]
 is nonincreasing. In particular, for every $t>0$ we obtain
 \begin{equation}\label{bound:xet}
  \norm[\Labi]{\f{\xe_{uu}(\cdot,t)-f(u)}{(\xe_u(\cdot,t))^2+ε}}=\norm[\Labi]{\xe_t(\cdot,t)}\le \norm[\Labi]{\f{{x_0}_{uu}-f(u)}{{x_0}_u^2}}.
 \end{equation}
\end{lemma}
\begin{proof}
 According to \eqref{regularity:xe} and \eqref{xepseq}, for every $τ\ge 0$ the function $\ye:=\xe_t$ satisfies 
 \begin{align*}
  &\yet=\f1{\xe_u^2+ε} \ye_{uu} -\f{2\xe_u(\xe_{uu}-f(u))}{(\xe_u^2+ε)^2} \ye_u \qquad &&\text{in }(a,b)\times (τ,\infty),\\
  &\ye(a,t)=0=\ye(b,t)\qquad &&\text{for all } t\in(τ,∞),\\
  &\ye(\cdot,τ)=\xe_t(\cdot,τ)=\f{\xe_{uu}(\cdot,τ)-f(u)}{\xe_u^2(\cdot,τ)+ε}.
 \end{align*}
 Comparison with the super- and subsolutions $\pm \norm[\Labi]{\ye(\cdot,τ)}$
 proves the assertion; the last part of the statement relies on the choice $τ=0$, \eqref{regularity:xe} and the obvious estimate 
 \[
  \norm[\Labi]{\f{{x_0}_{uu}-f(u)}{{x_0}_u^2+ε}}\le\norm[\Labi]{\f{{x_0}_{uu}-f(u)}{{x_0}_u^2}}.\qedhere
 \]
\end{proof}

\begin{lemma}\label{lem:energy}
 Let $a,b\in ℝ$ with $a<b$, $f\in C^2(ℝ)$ and let $x_0$ be as in \eqref{cond:x0} and $ε>0$. Then $\xe$ from Lemma~\ref{lem:exeps} satisfies 
 \[
  \ddt \kkl{\iab \f{\xe_u^2(u,t)}2\du  + \iab f(u)\xe(u,t)\du } \le -\iab \xe_u^2\xe_t^2\qquad \text{for } t\in (0,∞). 
 \]
\end{lemma}
\begin{proof}
 The boundary conditions in \eqref{xepseq} ensure that $\xe_t(a,t)=0=\xe_t(b,t)$ for all $t>0$, and hence 
a straightforward computation reveals 
 \begin{align*}
   \ddt &\kkl{\iab \f{\xe_u^2(u,t)}2\du  + \iab f(u)\xe(u,t)\du } = \iab \xe_u\xe_{ut} + \iab f(u)\xe_t \\
   &= -\iab \xe_{uu}\xe_t + \xe_u(b)\xe_t(b) - \xe_u(a)\xe_t(a) + \iab f(u)\xe_t\\
   &= -\iab (\xe_{uu}-f(u))\xe_t
    -\iab (\xe_u^2+ε)\xe_t^2
   \le -\iab \xe_u^2\xe_t^2 \quad \text{on } (0,∞). \qedhere
 \end{align*}
\end{proof}

Next we prepare a variant of Poincar\'e's inequality for elements of $X$.
\begin{lemma}\label{lem:almostpoincare}
 Let $a,b\in ℝ$ with $a<b$. Then 
  \[
  \betrag{\xi(u)} \le \sqrt{b-a} \kl{\iab(\xi_u(u))^2\du}^{\f12}\qquad \text{for } ξ\in X. 
 \]
\end{lemma}
\begin{proof}
 In light of the left boundary condition in \eqref{def:X}, the upper bound is an immediate consequence of the fundamental theorem of calculus and Hölder's inequality: 
 \begin{align*}
  \xi(u)&=\xi(a)+\int_a^u \xi(v)\dv \\
  &\le -1 +\kl{\int_a^u \xi_u^2(v) \dv}^{\f12} \kl{\int_a^u 1}^{\f12} \le \sqrt{b-a} \kl{\iab \xi_u^2(v)\dv}^{\f12}.
 \end{align*}
 In the same way, a corresponding lower bound results from 
 \begin{align*}
  \xi(u)&=\xi(b)-\int_u^b \xi(v)\dv \ge - \sqrt{b-a} \kl{\iab \xi_u^2(v)\dv}^{\f12}.\qedhere
 \end{align*}
\end{proof}

This inequality ensures that for $\calF$ from \eqref{def:F}, bounds on $\calF(ξ)$ already entail certain bounds for $ξ\in X$. 

\begin{lemma}\label{lem:Fcoercive}
 Let $a,b\in ℝ$ with $a<b$ and $f\in C^2(ℝ)$. For every $M>0$ there is $C>0$ such that the following implication holds: 
 Whenever $ξ\in X$ satisfies $\calF(ξ)\le M$, then 
 \begin{align}
  \norm[\Lab2]{ξ_u}&\le C \label{bound:xeuL2-prelim}\\
  \quad \text{and}\quad \norm[\Labi]{ξ}&\le C. \label{bound:xeLi-prelim}
 \end{align}
\end{lemma}
\begin{proof}
 If we abbreviate $I:=\iab ξ_u^2(u)\du$, then Lemma~\ref{lem:almostpoincare} implies 
 \[
  M\ge \f12 I - c_1 I^{\f12}, 
 \]
 where $c_1:=\sqrt{b-a}\iab |f(u)|\du$. Accordingly, $I\le 2c_1^2+2M+2c_1\sqrt{c_1^2+2M}$ and \eqref{bound:xeuL2-prelim} follows, so that \eqref{bound:xeLi-prelim} is a consequence of Lemma~\ref{lem:almostpoincare}. 
\end{proof}

An immediate consequence are the following bounds for solutions to \eqref{xepseq}.
\begin{lemma}\label{lem:boundsxe:h1andli}
 Let $a,b\in ℝ$ with $a<b$, $f\in C^2(ℝ)$ and let $x_0$ be as in \eqref{cond:x0}.
 There is $C>0$ such that for every $ε>0$ and $t>0$ the functions $\xe$ from Lemma~\ref{lem:exeps} fulfil 
 \begin{align}
  \norm[\Lab2]{\xe_u(\cdot,t)}&\le C \label{bound:xeuL2},\\
\norm[\Labi]{\xe(\cdot,t)}&\le C. \label{bound:xeLi}
 \end{align}
\end{lemma}
\begin{proof}
For every $ε>0$ and $t>0$, $\xe(\cdot,t)\in X$. Due to Lemma~\ref{lem:energy}, Lemma~\ref{lem:Fcoercive} hence becomes applicable with $M:=\iab\f{{x_0}_u^2}2 + \iab f(u)x_0$, and \eqref{bound:xeuL2} and \eqref{bound:xeLi} follow. 
\end{proof}

With the help of Lemma~\ref{lem:bound:xet}, we can turn these bounds into higher regularity information.
\begin{lemma}\label{lem:boundsxe:h2andcalpha}
 Let $a,b\in ℝ$ with $a<b$, $f\in C^2(ℝ)$ and let $x_0$ be as in \eqref{cond:x0} and $ε_0>0$. Then for every $p\in[1,∞)$ 
 there is $C>0$ such that the functions $\xe$ from Lemma~\ref{lem:exeps} satisfy 
 \begin{equation}\label{bound:xeuuLp}
  \norm[\Lab p]{\xe_{uu}(\cdot,t)} \le C \qquad \text{for all } t>0 \text{ and } ε\in(0,ε_0)
 \end{equation}
 and, with $α:=1-\f1p\in[0,1)$,   
 \begin{equation}\label{bound:xeuCalpha}
  \normm{C^{α}([a,b])}{\xe_u(\cdot,t)}\le C \qquad \text{for all } t>0 \text{ and } ε\in(0,ε_0).
 \end{equation}
\end{lemma}
\begin{proof}
 Due to Lemma~\ref{lem:bound:xet}, we can find $c_0>0$ such that for every $ε>0$, $t\in(0,∞)$ and almost every $u\in(a,b)$ 
we have
 \begin{equation}\label{eq:c0:forboundxeuuLp}
  \betrag{\xe_{uu}(u,t)-f(u)} \le c_0 \betrag{\xe_u^2+ε}. 
 \end{equation}
 Integration over $(a,b)$ combined with \eqref{bound:xeuL2} shows \eqref{bound:xeuuLp} for $p=1$. As $W^{1,1}((a,b))$ is compactly embedded into $L^q((a,b))$ for any $q\in[1,∞)$, from \eqref{bound:xeuuLp} for $p=1$ and \eqref{bound:xeuL2} we can conclude the existence of $c_1=c_1(q)>0$ such that 
 \[
  \norm[\Lab q]{\xe_u(\cdot,t)} \le c_1(q) \qquad \text{for all } ε\in(0,ε_0) \text{ and all } t>0.
 \]
 Once more relying on \eqref{eq:c0:forboundxeuuLp}, we see that hence for any $p\in(1,∞)$ 
 \[
  \iab \betrag{\xe_{uu}(u,t)-f(u)}^p\du \le c_0 \iab \betrag{\xe_u^2(u,t) +ε}^p \du \le 2^pc_0c_1(2p) + 2^p(b-a)ε_0^p 
 \]
 holds for every $t>0$ and $ε\in(0,ε_0)$ and therefore \eqref{bound:xeuuLp} is proven for arbitrary $p$. Obtaining \eqref{bound:xeuCalpha} is again possible by a Sobolev embedding. 
\end{proof}

Now we are ready to use compactness arguments to supply a solution to \eqref{xeq}.

\begin{lemma}\label{lem:conv}
 Let $a,b\in ℝ$ with $a<b$, $f\in C^2(ℝ)$ and let $x_0$ be as in \eqref{cond:x0}, $p\in(1,∞)$ and $α\in(0,1)$. 
Then there are a sequence $(ε_k)_{k\in ℕ}\searrow 0$ and a function 
\[
  x\in C^0([0,∞);C^{1+α}([a,b]))
\]
with
\[
  x_{uu}\in L^p_{loc}([a,b]\times[0,∞)),\qquad
  x_t\in L^{∞}((a,b)\times(0,∞)),
\]
 such that 
 \begin{align}
  x_{ε_k}&\to x && \text{in } C^0_{loc}([0,∞);C^{1+α}([a,b])),\label{conv:hoelder}\\
  {x_{ε_k}}_{uu}&\weakto x_{uu} && \text{in } L^p_{loc}([a,b]\times[0,∞)),\label{conv:xuulp}\\
  {x_{ε_k}}_t&\weakstarto x_t && \text{in } L^{∞}((a,b)\times(0,∞))\label{conv:xtli}, 
 \end{align}
 where $\xe$ is as given by Lemma~\ref{lem:exeps} for any $ε=ε_k$.
Moreover,
it holds that
 \begin{equation}\label{equation:weakly}
  \int_0^{∞}\iab x_u^2x_tφ = -\int_0^{∞}\iab x_uφ_u - \int_0^{∞}\iab f(u)φ
\qquad\text{for } φ\in C_c^\infty((a,b)\times(0,∞)),
 \end{equation}
 and $x(a,t)=-1$, $x(b,t)=1$ for every $t>0$, as well as $x(\cdot,0)=x_0$ in $(a,b)$. 
\end{lemma}
\begin{proof}
 Without loss of generality, we may assume $p$ to be large. Then $W^{2,p}((a,b))\hookrightarrow C^{1+α}([a,b])\hookrightarrow L^{p}((a,b))$, where the first embedding is compact. 
 As $\set{\xe: ε\in(0,1)}$ is bounded in $L^{∞}((0,∞);W^{2,p}((a,b)))$ according to \eqref{bound:xeuuLp} and \eqref{bound:xeLi}, and $\set{\xet: ε\in(0,1)}$ is bounded in $L^p((0,∞);L^p((a,b)))$ due to \eqref{bound:xet}, we can apply the Aubin--Lions lemma in the form of \cite[Cor. 8.4]{simon} to conclude that the set $\set{\xe: ε\in(0,1)}$ is compact in $C^0_{loc}([0,∞);C^{1+α}([a,b]))$ and extract a sequence such that \eqref{conv:hoelder} holds. Subsequently, \eqref{bound:xeuuLp} and \eqref{bound:xet} enable us to successively find subsequences along which \eqref{conv:xuulp} and \eqref{conv:xtli} are satisfied. 
 
 Because for every $ε>0$ we have 
 \[
    ε\int_0^{∞}\iab \xe_u^2\xe_tφ+ \int_0^{∞}\iab \xe_u^2\xe_tφ = -\int_0^{∞}\iab \xe_u φ_u - \int_0^{∞}\iab f(u)φ  
 \]
 for every $φ\in C_c^{∞}((a,b)\times(0,∞))$, as well as $\xe(a,t)=-1$, $\xe(b,t)=1$ for all $t>0$ and $\xe(u,0)=x_0(u)$ for all $u\in(a,b)$, \eqref{conv:hoelder} and \eqref{conv:xtli} also imply the remaining part of the statement upon a passage to the limit $ε=ε_k\searrow 0$.  
\end{proof}

\begin{proof}[Proof of Theorem~\ref{thm:existence}]
 The theorem is part of Lemma~\ref{lem:conv} and hence already proven. 
\end{proof}

\section{Long-time behaviour. Proof of Theorem~\ref{thm:longterm}} \label{sec:longterm}

In a first step we carry over the content of Lemma~\ref{lem:energy} from $\xe$ to $x$. 
\begin{lemma}\label{lem:spatiotemporalintegralfinite}
Let $a,b\in ℝ$ with $a<b$, $f\in C^2(ℝ)$ and let $x_0$ be as in \eqref{cond:x0}, and $t_1,t_2\in[0,∞)$, $t_1<t_2$. Then the function $x$ from Lemma~\ref{lem:conv} satisfies 
\[
 \f12\iab x_u^2(\cdot,t_1) + \iab f(u)x(u,t_1)\du + \int_{t_1}^{t_2}\iab x_u^2x_t^2 \le \f12 \iab x_u^2(\cdot,t_2) + \iab f(u)x(u,t_2)\du.
\]
\end{lemma}
\begin{proof}
 According to \eqref{conv:hoelder}, for $i\in\set{1,2}$, 
 \[
  \f12\iab \xe_u^2(\cdot,t_i) + \iab u\xe(u,t_i)\du \to \f12\iab x_u^2(\cdot,t_i) + \iab ux(u,t_i)\du
 \]
 as $ε=ε_k\searrow 0$. Moreover, \eqref{conv:hoelder} and \eqref{conv:xtli} together imply that $\xe_u\xe_t\weakto x_ux_t$ in $L^2((a,b)\times(t_1,t_2))$, so that 
 \[
  \int_{t_1}^{t_2}\iab x_u^2x_t^2\le \liminf_{ε_k\searrow 0} \int_{t_1}^{t_2}\iab \xe_u^2\xe_t^2. 
 \]
 The assertion therefore results from Lemma~\ref{lem:energy}.
\end{proof}


\begin{lemma}\label{lem:intfinite:xuu}
 Let $a,b\in ℝ$ with $a<b$, $f\in C^2(ℝ)$ and let $x_0$ be as in \eqref{cond:x0}. Then for the function $x$ from Lemma~\ref{lem:conv} we have 
 \begin{equation}\label{finiteintegral}
  \int_0^{∞} \kl{\iab \betrag{x_{uu}(u,t)-f(u)} \du}^2 dt <\infty. 
 \end{equation}
\end{lemma}
\begin{proof}
 From \eqref{conv:xuulp}, \eqref{conv:hoelder}, \eqref{conv:xtli} and \eqref{xepseq}, we may conclude that 
\[
x_{uu}(u,t)-f(u)=x_u^2(u,t)x_t(u,t)
\]
for almost every $(u,t)\in (a,b)\times(0,∞)$. Fixing $c_0>0$ such that 
 \[
  \iab x_u^2(u,t)\du \le c_0 \qquad \text{for every } t>0, 
 \]
 which is possible according to \eqref{bound:xeuL2} and \eqref{conv:hoelder}, we observe that by Hölder's inequality 
 \begin{align*}
  \int_0^{∞} \kl{\iab \betrag{x_{uu}(u,t)-f(u)} \du}^2\dt &= \int_0^{\infty}\kl{\iab x_u^2(u,t)|x_t(u,t)| \du}^2\dt\\
  &\le \int_0^{∞}\iab x_u^2(u,t) x_t^2(u,t)\du \iab x_u^2(u,t)\du \dt\\
  &\le c_0 \int_0^{∞}\iab x_u^2(u,t) x_t^2(u,t)\du \dt, 
 \end{align*}
 which is finite as a consequence of Lemma~\ref{lem:spatiotemporalintegralfinite}.
\end{proof}

 The fact that the integral in \eqref{finiteintegral} is finite reveals the existence of an unbounded sequence of times along which the solution converges. 

\begin{lemma}\label{lem:minimizingtimesequence}
 Let $a,b\in ℝ$ with $a<b$, $f\in C^2(ℝ)$ and let $x_0$ be as in \eqref{cond:x0}. With $x^{∞}$ as in \eqref{def:xinfty-general},  
 there is a sequence $(τ_k)_{k\inℕ}\nearrow ∞$ such that $x$ given by Lemma~\ref{lem:conv} satisfies 
 \[
  x(\cdot,τ_k)\to x^{∞} \qquad \text{ in } C^1([a,b])
 \]
 as $k\to \infty$.
\end{lemma}
\begin{proof}
 Because of \eqref{bound:xeLi}, \eqref{bound:xeuCalpha} and \eqref{conv:hoelder}, there is $c_0>0$ such that, for some $α\in(0,1)$,  
 \begin{equation}\label{eq:hoelderbound:x}
  \normm{C^{1+α}([a,b])}{x(\cdot,t)} \le c_0 \qquad \text{for every } t\in (0,∞).
 \end{equation}
 Furthermore, for this proof, by abuse of notation, we consider $x_{uu}$ as one fixed representative of $x_{uu}$ (which actually is an equivalence class with respect to equality almost everywhere). Then there is a set $N\subset (0,∞)$ of measure zero such that for every $t\in (0,∞)\setminus N$, $x_{uu}(\cdot,t)$ is the weak $u$-derivative of $x_u(\cdot,t)$. 
 In view of Lemma~\ref{lem:intfinite:xuu}, we can find a sequence $(t_k)_{k\in ℕ}\subset (0,∞)\setminus N$ such that 
 \begin{equation}\label{eq:conv:xuu}
t_k\to\infty \quad \text{and}\quad \iab \betrag{x_{uu}(u,t_k)-f(u)} \du \to 0\qquad\text{ as } k\to∞.
 \end{equation}
 If we use \eqref{eq:hoelderbound:x} together with the Arzel\`a--Ascoli theorem and combine it with \eqref{eq:conv:xuu}, we see that (for a subsequence of $(τ_k)_{k\in ℕ}\subset (t_k)_{k\in ℕ}$) 
 \[
  x(\cdot,τ_k)\to x^{\infty} \qquad \text{ in } C^1([a,b]) \text{ and } W^{2,1}((a,b)) \qquad\text{ as } k\to\infty
 \]
 with some $x^{∞}\in W^{2,1}((a,b))$, whose weak second derivative satisfies $x^{∞}_{uu}(u)=f(u)$ for almost every $u\in(a,b)$. Finally taking into account that $x(a,τ_k)=-1$, $x(b,τ_k)=1$ for every $k\in ℕ$, we can compute $x^{∞}$ from $f$, arriving at \eqref{def:xinfty-general}. 
\end{proof}

In order to transform ``convergence along one particular sequence $τ_k\nearrow ∞$'' to ``convergence as $t\to \infty$'', we identify $x^{∞}$ as the global minimizer of $\calF$. 

\begin{lemma}\label{lem:uniqueminimizer}
 Assume that $a,b\in ℝ$ with $a<b$, $f\in C^2(ℝ)$ and $X$ is as in \eqref{def:X}. 
 Then $\calF$ from \eqref{def:F} has a unique global minimum in $X$, and the minimizer coincides with $x^{∞}$ from \eqref{def:xinfty-general}. 
\end{lemma}
\begin{proof}
 As $\calF(ξ)\ge \f12 I - c_1\sqrt{I}$ with $I=\iab ξ_u^2$ and $c_1=\sqrt{b-a}\iab |f(u)|\du$ (according to Lemma~\ref{lem:almostpoincare}) and $\inf\set{\f12 I-c_1\sqrt{I}\mid I\ge 0}>-\infty$, $\calF$ is bounded from below. Letting $(ξ_k)_{k\in ℕ}\subset X$ be a sequence such that $\calF(ξ_k)\to \inf_X \calF$, we can conclude from Lemma~\ref{lem:Fcoercive} that $(ξ_k)_{k\in ℕ}$ is bounded in $W^{1,2}((a,b))$ and hence along a subsequence weakly converges to some $ξ\in W^{1,2}((a,b))$. Moreover, $ξ\in X$, because $X$ is weakly closed. Since $\calF$ is convex and continuous and hence certainly weakly sequentially lower-semicontinuous, $\calF(ξ)\le \liminf_{k\to\infty} \calF(ξ_k)$, and $ξ$ is a global minimizer. 
 In particular, for every $η\in ℝ$ and $φ\in C_c^{∞}((a,b))$, $ξ+ηφ\in X$ and $\calF(ξ+ηφ)\ge \calF(ξ)$. 
 Written explicitly, that means 
\[
 η^2\iab\f{φ_u^2}2+η\iabξ_uφ_u +η\iabφf(u)\ge 0
\]
and division by $η>0$ and subsequently taking the limit $η\to 0$ show that $ξ_{uu}=f(u)$ in the weak sense. Together with $ξ\in X$, this ensures $ξ=x^{\infty}$. 
\end{proof}

We finally want to see that Lemma~\ref{lem:minimizingtimesequence} in conjunction with Lemma~\ref{lem:uniqueminimizer} ensures $x(\cdot,t)\to x^{∞}$ as $t\to \infty$. In a slightly different, but equivalent formulation this is the assertion of Lemma~\ref{lem:teilfolgevonteilfolge}:

\begin{lemma}\label{lem:teilfolgevonteilfolge}
 Let $a,b\in ℝ$ with $a<b$, $f\in C^2(ℝ)$ and let $x_0$ be as in \eqref{cond:x0} and $(t_k)_{k\in ℕ}\nearrow ∞$. Then there is a subsequence $(t_{k_j})_{j\in ℕ}$ of $(t_k)_{k\inℕ}$ such that 
 \[
  x(\cdot,t_{k_j})\to x^{∞} \qquad \text{in } C^1([a,b]), 
 \]
 where $x$ is taken from Lemma~\ref{lem:conv} and $x^{∞}$ is as defined in \eqref{def:xinfty-general}.
\end{lemma}
\begin{proof}
 Since \eqref{bound:xeLi} and \eqref{bound:xeuCalpha} together with \eqref{conv:hoelder} ensure existence of $α\in(0,1)$ and $c_0>0$ such that 
 \[
  \normm{C^{1+α}([a,b])}{x(\cdot,t_k)}\le c_0 \qquad \text{for every } k\in ℕ, 
 \]
 the theorem of Arzel\`a--Ascoli ensures the existence of a subsequence $(t_{k_j})_{j\in ℕ}$ such that $x(\cdot,t_{k_j})\to \xhat$ in $C^1([a,b])$ for some $\xhat\in C^1([a,b])\cap X$. This subsequence can be chosen such that $τ_j\le t_{k_j}$ for every $j\in ℕ$, where $(τ_j)_{j\in ℕ}$ is the sequence provided by Lemma~\ref{lem:minimizingtimesequence}. Accordingly, $\calF(x(\cdot,t_{k_j}))\le \calF(x(\cdot,τ_j))$ due to Lemma~\ref{lem:intfinite:xuu} and thus 
 \[
  \calF(\xhat)=\lim_{j→∞} \calF(x(\cdot,t_{k_j}))\le \lim_{j→∞}\calF(x(\cdot,τ_j)) = \calF(x^{∞}). 
 \]
 Lemma~\ref{lem:uniqueminimizer} serves to show $\xhat=x^{∞}$. 
\end{proof}

\begin{proof}[Proof of Theorem~\ref{thm:longterm}]
 As the solution to \eqref{xeq} from Theorem~\ref{thm:existence} is the function from Lemma~\ref{lem:conv}, Theorem~\ref{thm:longterm} immediately follows from Lemma~\ref{lem:teilfolgevonteilfolge}. 
\end{proof}

\section{Where can the spatial derivative vanish? Proof of Theorem~\ref{thm:bdry}}\label{sec:where}

In order to see where blow-up in \eqref{ueq} occurs or where the coefficient $\f1{x_u^2}$ in \eqref{xeq} becomes singular, it is crucial to identify possible zeros of $x_u$. The following lemma will be an important component of the proofs of Theorems \ref{thm:bdry} and \ref{thm:special}. 

\begin{lemma}\label{lem:firstatzero}
 Let $a, b \in ℝ$, $a<b$ and $α\in[a,b)$. Let $f\in C^2(ℝ)$ be positive on $(α,b]$. Assume that $x_0$ satisfies \eqref{cond:x0} and let $x$ be the solution to \eqref{xeq} from Theorem~\ref{thm:existence} and suppose that $τ\mapsto x(α,τ)$ is constant.
 If $t$ is such that $x_u(α,t)>0$, then $x_u>0$ in all of $(α,b)$. 
\end{lemma}
\begin{proof}
 We let $c_0$ be such that $|\xe_t(u,τ)|\le c_0$ for all $(u,τ)\in[0,b]\times(0,∞)$ (cf. \eqref{bound:xet}). 
 Then we choose $δ>α$ and $c_1>0$ such that $x_u(u,t)>c_1$ for all $u\in[α,δ]$ and let $δ_1:=\inf_{[δ,b]}f$, so that the assumptions on $f$ imply $δ_1>0$. 
 We use \eqref{conv:hoelder} to find $ε_0\in(0,\f{c_0}{2δ_1})$ such that for all $ε\in(0,ε_0)$, $\xe_u(δ,t)>\f{c_1}2$. 
 For any $ε\in(0,ε_0)$ and $η\in(0,\sqrt{\f{c_0}{2δ_1}})\cap(0,\f{c_1}2)$, we let 
 $y(u):=y_{ε,η}(u):=\xe(u,t)-ηu$, $u\in[δ,b]$. 
 We then use \eqref{xepseq} to compute 
 \[
  y_{uu}=\xe_{uu} = (\xe_u^2(u,t)+ε)\xe_t + f(u) = ((y_u+η)^2+ε)\xe_t +f(u),\qquad u\in[δ,b].
 \]
 If we assume that for some $u_0\in (δ,b)$ we had $y_u(u_0)=0$, then 
 \[
  y_{uu}(u_0) = f(u_0) + (η^2+ε)\xe_t(u_0,t) \ge δ_1 - (η^2+ε) c_0 >0.
 \]
 Therefore, every critical point of $y$ in $(δ,b)$ would have to be a local minimum, which is not possible (note that $y_u(δ)>0$), and we can conclude that $y_u>0$ on all of $[δ,b)$. Passing to the limit $ε\searrow 0$, due to \eqref{conv:hoelder}, $x_u(\cdot,t)\ge η$ in $[δ,b]$ and by the definition of $δ$, this proves positivity of $x_u(\cdot,t)$ in $[α,b]$. 
\end{proof}

\begin{proof}[Proof of Theorem~\ref{thm:bdry}]
 According to \eqref{def:xinfty-general}, we have 
 \[
  x^{∞}(u) = \f{u^2}{2}+ \f{4-b^2}{2b} u -1, \qquad u\in[0,b].
 \]
 The choice of $b>2$ hence implies $x_u^{∞}(0)<0$ and \eqref{eq:thm:conv} thus ensures that $t_0$ is finite. With $α=0$, we can apply Lemma~\ref{lem:firstatzero}, immediately obtaining \eqref{eq:nullstellevonxu}.
\end{proof}

\section{A symmetric setting. Proof of Theorem~\ref{thm:special}} \label{sec:special}

In this whole section, we work toward the proof of Theorem~\ref{thm:special}, that is, we deal with the symmetric setting of said theorem. The first elementary observation is that the solution retains the symmetry. Moreover, we can restrict arguments to the half $(0,b)$ of the domain, still retaining useful boundary conditions at $u=0$. 

\begin{lemma}
 Let $-a=b\in(0,∞)$, $f(u)\equiv u$, and assume that $x_0$ satisfies \eqref{cond:x0} and \eqref{cond:x0:special}. 
Then the solution $x$ to \eqref{xeq} given by Theorem~\ref{thm:existence} satisfies $x(u,t)=-x(-u,t)$ for every $u\in(-b,b)$ and $t>0$. In particular,  
 \begin{align}
  x(0,t)&=0 \qquad \text{ for all } t>0\label{eq:special:xzero}\\
 \text{and }\; x_t(0,t)&=0 \qquad \text{ for all } t>0.\label{eq:special:xtzero}
 \end{align} 
\end{lemma}
\begin{proof}
By uniqueness of the solutions to \eqref{xepseq}, $\xe(u,t)=-\xe(-u,t)$ for all $t>0$ and hence \eqref{conv:hoelder} ensures that also $x(\cdot,t)$ is odd for every $t>0$. This implies $x(0,t)=0$ for all $t>0$, which in turn proves \eqref{eq:special:xtzero}. 
\end{proof}

At each $u\in[0,b)$, the solution $\xe(u,t)$ decreases with respect to $t$ (and accordingly increases in $(-b,0]$).

\begin{lemma}\label{lem:decreasingforpositivearguments}
 Let $-a=b\in(0,∞)$, $f(u)\equiv u$, and assume that $x_0$ satisfies \eqref{cond:x0} and \eqref{cond:x0:special}. 
 For every $ε>0$, the function $\xe(u,\cdot)$ from Lemma~\ref{lem:exeps} is nonincreasing for every $u\in[0,b)$. 
\end{lemma}
\begin{proof}
 If we let $y_{0ε}(u):=\f{x_{0uu}(u)-u}{x_{0u}^2(u)+ε}$, then $y_{0ε}(u)\le 0$ for every $u\in(0,b)$,  and $\ye:=\xe_t$ satisfies 
 \begin{align*}
  \yet &= \f1{\xe_u^2+ε} \ye_{uu} - \f{2\xe_u(\xe_{uu}-u)}{(\xe_u^2+ε)^2} \ye_u\quad \text{in } (0,b)\times (0,∞),\\
  \ye(0,t)&=0=\ye(b,t)\qquad \text{for all } t>0,\qquad \ye(\cdot,0)=y_{0ε} \text{ in }[0,b].
 \end{align*}
 Here, apparently, the left boundary condition relies on \eqref{eq:special:xtzero}. 
 Comparison with the upper solution $0$ shows that $\xe_t$ is nonpositive throughout $(0,b)\times(0,∞)$. 
\end{proof}

Together with \eqref{eq:special:xzero}, this has an implication on the spatial derivative $x_u$ at $0$: It cannot increase over time.

\begin{lemma}\label{lem:neveragaininvertible}
 Let $-a=b\in(0,∞)$, $f(u)\equiv u$, and assume that $x_0$ satisfies \eqref{cond:x0} and \eqref{cond:x0:special}. With $x$ being the solution to \eqref{xeq} from Theorem~\ref{thm:existence}, the map 
 \[(0,∞)\ni t\mapsto x_u(0,t)\]
 is nonincreasing. In particular: If $x_u(0,t_0)\le 0$ for some $t_0\in (0,∞)$, then $x_u(0,t)\le 0$ for every $t\ge t_0$ and $x(\cdot,t)$ has no differentiable inverse for any $t\ge t_0$.
\end{lemma}
\begin{proof}
 We know that for every $t>0$, $x(\cdot,t)$ is odd and hence $x(0,t)=0$. Since furthermore $x(u,t)\le x(u,t+τ)$ for any $u\in(0,b)$ and $τ>0$ by Lemma~\ref{lem:decreasingforpositivearguments} and \eqref{conv:hoelder}, we have that 
 \[
  x_u(0,t)=\lim_{u\searrow 0^+} \f{x(u,t)}{u} \le \lim_{u\searrow 0^+} \f{x(u,t+τ)}{u} = x_u(0,t+τ).
 \]
 If for some $t>0$, $x_u(0,t)<0$, then -- due to $x(a,t)<x(b,t)$ -- $x(\cdot,t)$ is not injective and hence not invertible. If $x_u(0,t)=0$, then even if $x(\cdot,t)$ is invertible, its inverse will not be differentiable at $0=x(0,t)$. 
\end{proof}

\begin{proof}[Proof of Theorem~\ref{thm:special}]
 For this choice of $f$, $x^{∞}$ is given explicitly by \eqref{def:xinfty-example}. If $b>\sqrt[3]6$, then $x_u^{∞}(0)=\f{6-b^3}{6b}<0$ and hence according to Theorem~\ref{thm:longterm}, the last part of the statement is proven and, furthermore, the set 
 \[
  S:=\set{t\in(0,∞)\mid \exists u\in (-b,b):\;x_u(u,t)<0}
 \]
 is nonempty. If we let $t_0:=\inf S$, then \eqref{cond:x0} together with the regularity assertion $x\in C^0([0,∞);C^{1+α}([a,b]))$ from Theorem~\ref{thm:existence} ensures positivity of $t_0$. 
 Moreover, Lemma~\ref{lem:firstatzero}, applied for $α=0$, shows that $x_u(0,t_0)=0$ and Lemma~\ref{lem:neveragaininvertible} guarantees that for no $t\ge t_0$ the function $x(\cdot,t)\colon [-b,b]\to [-1,1]$ has a differentiable inverse. 
\end{proof}

\appendix

\section{Equivalence of the problems}\label{sec:equivalence}
A central observation this work relies on is that the problems \eqref{ueq} and \eqref{xeq} can be transformed into each other. We therefore give some details concerning this equivalence, even though the arguments are rather elementary. 
\subsection{Turning $u$ into $x$}\label{sec:u-to-x}
Assume that $b>a$, $T>0$, $u_0$ is monotone increasing with $u_0(-1)=a$, $u_0(1)=b$ and $u\in C^{4,2}([-1,1]\times(0,T))\cap C([-1,1]\times[0,T))$ solves \eqref{ueq}. Then, by comparison (\cite[Prop. 52.7]{QS}) with the constant solutions $a$ and $b$, respectively, we obtain that $a\le u\le b$ in $(-1,1)\times(0,T)$. Accordingly, $v:=u_x$ satisfies $v\ge 0$ on $\set{-1,1}\times(0,T)$ as well as on $(-1,1)\times\set{0}$. As $v$ additionally solves 
\[
 v_t = v_{xx} + 3f(u)u_x^2v_x + u_x^3f'(u) v,
\]
the strong maximum principle (see \cite[Prop. 52.7]{QS}) readily implies $v>0$ in $(-1,1)\times(0,T)$. Therefore, for every $t\in [0,T)$, the function $u(\cdot,t)\colon [-1,1]\to [a,b]$ is invertible and we can define 
\[
 x(\cdot,t) := \kkl{u(\cdot,t)}^{-1}\colon [a,b]\to[-1,1],\qquad \text{for every } t\in[0,T). 
\]
This function is differentiable with respect to $u\in(a,b)$, and $x_u(u,t)=\f1{u_x(x(u,t),t)}$ for all $u\in(a,b), t\in[0,T)$. Regularity of $u$ and the Implicit Function Theorem furthermore ensure the existence of second spatial and first temporal derivatives of $x$, and it is easy to see that 
\begin{align*}
 0&=\f{∂^2}{∂ξ^2} ξ 
  = \f{∂^2}{∂ξ^2} x(u(ξ,t),t) 
 = x_u(u(ξ,t),t)u_{ξξ}(ξ,t) + x_{uu}(u(ξ,t),t)u_{ξ}^2(ξ,t) 
\end{align*}
and hence 
\[
 u_{ξξ}(ξ,t)=\f{-x_{uu}(u(ξ,t),t) u_{ξ}^2(ξ,t)}{x_u(ξ,t)}\qquad \text{for every } ξ\in(-1,1), ~t\in[0,T). 
\]
Similarly, for all $(ξ,t)\in(-1,1)\times(0,T)$,
\[
 0=\ddt ξ = \ddt x(u(ξ,t),t) = x_u(u(ξ,t),t)u_t(ξ,t) + x_t(u(ξ,t),t), 
\]
i.e. 
\begin{align*}
 x_t(u(ξ,t),t)
  &=-x_u(u(ξ,t),t)u_t(ξ,t) = -x_u(u(ξ,t),t) \kl{u_xx(ξ,t) + f(u(ξ,t)) u_{ξ}^3(ξ,t)} \\
  &= -x_u(u(ξ,t),t) \kl{\f{-x_{uu}(u(ξ,t),t) u_{ξ}^2(ξ,t)}{x_u(ξ,t)} + \f{f(u(ξ,t))}{x_u^3(u(ξ,t),t)}}\\
 &= \f{x_{uu}(u(ξ,t),t) - f(u(ξ,t))}{x_u^2(u(ξ,t),t)}, 
\end{align*}
which means that 
\[
 x_t(u,t) = \f{x_{uu}(u,t) - f(u)}{x_u^2(u,t)}\qquad \text{for all } u\in (a,b),\; t\in(0,T).
\]

\subsection{Turning $x$ into $u$}\label{sec:x-to-u}
Let $x\in C^{2,1}([a,b]\times[0,T))$ satisfy $x_t=\f{x_{uu}-f(u)}{x_u^2}$, $x(a,\cdot)=-1$, $x(b,\cdot)=1$ and $x_u>0$ in $[a,b]\times[0,T)$ for some $T>0$. Then we can define 
\[
 u(\cdot,t):=\kkl{x(\cdot,t)}^{-1}, \qquad t\in [0,T), 
\]
and conclude from the Implicit Function Theorem and 
computations similar to those in Section \ref{sec:u-to-x} that $u\in C^{2,1}([-1,1]\times[0,T))$ and 
\[
 u_t(x,t)=u_{xx}(x,t)+f(u(x,t))u_x^3(x,t) \qquad \text{for all } x\in(-1,1),\;t\in(0,T).
\]

\textbf{Acknowledgements.} The first author was supported in part by the Slovak
Research and Development Agency under the contract No. APVV-14-0378 and by the VEGA grant
1/0347/18. 

{\footnotesize 
\def\cprime{$'$}

}

\end{document}